\documentclass[12pt]{amsart}
\usepackage{graphics,wrapfig, graphicx, mathrsfs,xcolor}
\usepackage{mathtools}
\usepackage{amssymb,amscd,graphicx,xcolor,subfig,tikz}
\usepackage{graphics}
\usepackage{indentfirst}
\usepackage{bm, enumerate,xcolor}
\usepackage[dvips]{epsfig}
\usepackage{latexsym}
\usepackage{float}
\usepackage{amsmath}
\usepackage{amssymb}
\usepackage[applemac]{inputenc}
\usepackage[T1]{fontenc}
\newtheorem{lemma}{Lemma}[section]
\newtheorem{theorem}{Theorem}[section]

\newtheorem{proposition}{Proposition}[section]

\newtheorem{definition}{Definition}[section]
\newtheorem{remark}{Remark}[section]

\numberwithin{equation}{section} \numberwithin{theorem}{section}
\numberwithin{example}{section} \numberwithin{remark}{section}
\numberwithin{figure}{section} \numberwithin{algorithm}{section}

\def\ba{\begin{array}}
\def\ea{\end{array}}
\def\bma{\left(\begin{matrix}}
\def\ema{\end{matrix}\right)}
\def\be{\begin{equation}}
\def\ee{\end{equation}}

\def\dfrac{\displaystyle\frac}

\mathtoolsset{showonlyrefs}

\def\ds{\displaystyle}
\title{Optimal H{\"o}lder convergence of a class of singular steady states to the Bahouri-Chemin patch}
\author{Yupei Huang}
\address{Department of mathematics, Duke University, Durham, NC}\email{yh298@duke.edu}
\author{Chilin Zhang}
\address{School of Mathematical Sciences, Fudan University, Shanghai 200433, China}\email{zhangchilin@fudan.edu.cn}

\begin{document}
\begin{abstract}
   Singular steady states are important objects in obtaining ill-posedness results for 2D incompressible Euler equations.  In \cite{elgindi2022regular}, a family of singular steady states near the Bahouri-Chemin patch was introduced. In this paper, we obtain the optimal convergence results for the singular steady states constructed in \cite{elgindi2022regular} to the Bahouri-Chemin patch. We first derive a boundary Harnack principle, and then obtain the optimal convergence results using the singular integral representation based on Green's function.
\end{abstract}
\maketitle

\section{Introduction}
The 2D incompressible Euler equations are a fundamental model to describe the motion of fluid:\begin{equation}\label{2D Euler}
    \begin{aligned}
        &\omega_t+u\cdot \nabla \omega=0,\\
        & u=\nabla^{\perp}\Delta^{-1}\omega,\text{where $\nabla^{\perp}=(-\partial_y,\partial_x)$,} 
    \end{aligned}
\end{equation}
here $u$ is the velocity of the fluid and $\omega$ represents the vorticity of the fluid. 
In this paper, we study the solutions of the 2D incompressible Euler equations on the flat torus $\mathbb{T}^2$. 
It is well known that if the initial vorticity is bounded, there exists a global-in-time solution to \eqref{2D Euler} and the vorticity stays bounded (\cite{YUDOVICH19631407}). Meanwhile, the Cauchy problem of the 2D incompressible Euler equations with more singular initial data remains a fundamental and interesting question, as many physical relevant data fall into this category (for example, vortex sheets \cite{wu2006mathematical}).\par 

It is a challenging problem to determine whether \eqref{2D Euler} is ill-posed when the initial vorticity lies in a more singular space.
In \cite{bourgain2015strong}, the authors proved the norm inflation for the 2D incompressible Euler equation when the initial vorticity lies in $H^1(\mathbb{T}^2)$.  Later, 
in \cite{elgindi2017ill}, the authors provided a simplified proof of this result. Subsequently, in \cite{jeong2021loss}, the author extended the norm inflation result to the class of vorticities in $W^{1,p}$ for $p>1$. A common feature of \cite{elgindi2017ill} and \cite{jeong2021loss} is that they studied the case where the initial vorticity was close to a specific singular steady state in $\mathbb{T}^2$, known as the Bahouri-Chemin patch. Their results rely heavily on the singular behavior of the velocity field of the Bahouri-Chemin patch. This approach motivates the study of singular steady states to obtain ill-posedness results for the 2D incompressible Euler equations.\par

To investigate steady states of the 2D incompressible Euler equations, 
a natural way is to introduce the stream function $\psi:=\Delta^{-1}\omega$. Then $\psi$ satisfies the equation: \begin{equation}\label{steady2}
\nabla^{\perp}\psi\cdot\nabla \Delta\psi=0. 
\end{equation}
There is a great variety of solutions to \eqref{steady2}, and an important family of them satisfies the semilinear elliptic equations $\Delta \psi =F(\psi).$ In fact, semilinear elliptic equations have been widely used to construct steady states of the 2D incompressible Euler equations (see \cite{Choffrut2012LocalSO}, \cite{constantin2021flexibility}, and \cite{Zelati2020StationarySN}). In this context, we also highlight the previous work in \cite{elgindi2022regular}. \par  
In the case of the Bahouri-Chemin patch, the stream function is given by $\psi_0(x_1,x_2)=\Delta^{-1}(-sgn(x_1)sgn(x_2))$. It is straightforward to verify that $\psi_0$ satisfies the following semilinear elliptic equation: \begin{equation}\label{BC EQU}
    \Delta\psi_0=-sgn(\psi_0).
\end{equation} 
In \cite{elgindi2022regular}, the authors introduced a singular approximation of \eqref{BC EQU} to construct singular steady states near the Bahouri-Chemin patch. 
More precisely, they introduced a symmetry class of functions in the first quadrant of torus $\mathbb{T}^{++}$ (see the precise definition in the Notation section). The symmetry class is defined as
\begin{align}
    \mathcal{K}:=&\{\phi|\phi(x_{1},x_{2})=-\phi(x_{1},-x_{2})=-\phi(-x_{1},x_{2})=\phi(x_{2},x_{1})\}\\
    &\cap\{\phi|\phi(x_{1},x_{2})\geq \psi_0(x_{1},x_{2}), \text{for $(x_{1},x_{2}) \in \mathbb{T}^{++}\}$}.
\end{align}
By studying the solutions $\phi_{\epsilon}\in\mathcal{K}$ to the semilinear elliptic equations:  \begin{equation}\label{algebraic}
  \Delta \phi_{\epsilon}=G_{\epsilon}(\phi_{\epsilon})\ \mbox{in}\ \mathbb{T}^{++},    
  \end{equation}
the authors constructed a sequence of singular steady states $\phi_{\epsilon}\in\mathcal{K}$ converging to $\psi_{0}$ in $C^{1}(\mathbb{T}^2)$. A key element in their work was the introduction of a barrier function. The particular choice of $G_{\epsilon}$ in \eqref{algebraic} was an odd function  satisfying \begin{align*}
    G_{\epsilon}(t)=&\left\{\begin{array}{cc}
        -\epsilon^{s}/t^{s}, &\quad 0<t<\epsilon,\\
         -1, &\quad t\geq \epsilon.
    \end{array}\right.
\end{align*} The specific form of $G_{\epsilon}$ inspired the authors to construct a barrier function $\Psi$, which is an odd-odd symmetric homogeneous solution of degree $\frac{2}{1+s}$ for the equation
\begin{equation}\label{homogeneous solution definition}
    \Delta\Psi_{\epsilon}=-\frac{\epsilon^{s}}{\Psi_{\epsilon}^{s}}\mbox{ in }\mathbb{R}^{2}.
\end{equation}
In \cite{elgindi2022regular}, the authors gave a lower bound for $\phi_{\epsilon}$ near the origin in $\mathbb{T}^{++}$ using $\Psi_{\epsilon}$:
\par 
\begin{theorem}\label{Barrier previous}
There is a constant $C$, such that for solutions $\phi_{\epsilon}$ to \eqref{algebraic} in $\mathcal{K}$, we have \begin{equation}
        \phi_{\epsilon}(x)\geq \Psi_{\epsilon}(x), \text{ for $x \in B_{C\sqrt{\epsilon}}\cap \mathbb{T}^{++}$.} 
    \end{equation}
\end{theorem}
We will recall the details of $\Psi_{\epsilon}$ in Lemma \ref{Barrier function} in the Appendix, and it will be used later. The specific expression of $G_{\epsilon}$, together with Theorem \ref{Barrier previous}, provides an upper bound of $\Delta(\phi_{\epsilon})$. Then, using standard potential theory, the authors proved that $\phi_{\epsilon}$ converges to $\psi_0$ in $C^{1}$.

\par
The purpose of this work is to investigate in greater detail the singular steady states constructed in \cite{elgindi2022regular}. The qualitative estimates of these singular steady states may be useful in proving ill-posedness results for the 2D Euler equations.\par 
We now present the statements of our main results.
\subsection{Main Results}
In this paper, we establish the following three main results concerning the solution $\phi_{\epsilon}$ of equation \eqref{algebraic}.  
First, we sharpen the estimate of the stream functions $\phi_{\epsilon}$ in Theorem \ref{Barrier previous}.
\\
    \begin{theorem}\label{phi with upper bound}
There exists a fixed constant $\sigma\in(0,\frac{s}{100})$ such that for sufficiently small $\epsilon$, and for any $x \in \mathbb{T}^{++}$ satisfying  
\[
    |x| \leq \frac{\sqrt{\epsilon}}{-\ln{\epsilon}},
\]
we have  
\begin{equation}
    \Psi_{\epsilon}(x) \leq \phi_{\epsilon}(x) \leq \Psi_{\epsilon}(x) \Big(1 + \epsilon^{-\sigma} |x|^{\sigma} \Big),
\end{equation}
where $\Psi_{\epsilon}$ is the barrier function satisfying \eqref{homogeneous solution definition}.
\end{theorem}
Based on Theorem \ref{phi with upper bound}, we derive the following asymptotic expansion of $\phi_{\epsilon}$.

\begin{theorem}\label{main1}
Let $\alpha_s = \frac{1-s}{1+s}$, for sufficiently small $\epsilon$, 
 there exist constants $C$ and $\sigma_{0}$ depending on $s$ such that  
\begin{equation}\label{Higher Holder Estimates1}
    \|\phi_{\epsilon} - \Psi_{\epsilon} - \psi_0\|_{C^{1,\alpha_s+\sigma_{0}}(\mathbb{T}^{2} \cap B_{\frac{1}{3}}(0))} \leq C \epsilon^{\frac{s}{4}},
\end{equation}
where $\psi_0$ is the stream function of the Bahouri-Chemin patch satisfying \eqref{BC EQU}.
\end{theorem}
Based on the quantitative estimates above, we now obtain sharp convergence results for our singular steady states with respect to the Bahouri-Chemin patch.

\begin{theorem}\label{main2}
   For sufficiently small $\epsilon$, we give the following results concerning $\phi_{\epsilon}$:
    \begin{itemize}
        \item \textbf{Continuity of $\phi_{\epsilon}$ with respect to $\epsilon$}:  
        $\phi_{\epsilon}$ is continuous in $\epsilon$ in the topology of $C^{1,\alpha_{s}}(\mathbb{T}^2)$, where $\alpha_s = \frac{1-s}{1+s}$.
        
        \item \textbf{Sharp convergence of $\phi_{\epsilon}$ to $\psi_0$}:  
        $\phi_{\epsilon} \notin C^{1,\beta}(\mathbb{T}^2)$ for any $\beta > \alpha_s$. Moreover, $\phi_{\epsilon}$ converges to $\psi_0$ in $C^{1,\alpha_s}(\mathbb{T}^2)$ as $\epsilon \to 0$.
    \end{itemize}
\end{theorem}

\subsection{The Key Method: the Boundary Harnack Principle}
The key method to obtain Theorem \ref{phi with upper bound} is to develop the following boundary Harnack principle for the semilinear elliptic equation $\Delta u = -u^{-s}$.
\begin{theorem}[Boundary Harnack Principle]\label{boundary Harnack with limit 1}
    Let $0 < s < 1$, and suppose $u \geq \Psi$ satisfies
    \begin{equation}
    \begin{aligned}
         &\Delta u = -u^{-s} \quad \text{in } B_{R} \cap \mathbb{R}^{++},\\
         &u(x_1, x_2) = 0, \quad \text{if } x_1 = 0 \text{ or } x_2 = 0.
    \end{aligned}
    \end{equation}
    Then, there exist constants $C, \sigma > 0$ depending only on $s$ such that for any $x \in B_{R/2} \cap \mathbb{R}^{++}$, we have 
    \begin{equation}
        \frac{u}{\Psi}(x) - 1 \leq C R^{-\frac{2}{1+s}} |x|^{\sigma} \cdot \max_{B_{R}} u.
    \end{equation}
    Here, $\Psi$ is a shorthand for $\Psi_{1}$, which is a non-negative homogeneous solution of $\Delta \Psi = -\Psi^{-s}$.
\end{theorem}
The idea was initially inspired by Krylov \cite{K83} and Allen-Shahgholian \cite{allen2019new}. For example, let $L = A^{ij} \partial_{ij}$ or $L = \mathrm{div}(A\nabla)$ be a uniformly elliptic operator. If $u$ satisfies
\begin{equation}
     Lu = f \in L^{\infty} \quad \text{in } B_{1}^{+}, \quad u \Big|_{x_{n}=0} = 0,
\end{equation}
then the ratio $u/x_n$ is H\"older continuous in $B_{1/2}^{+}$.\par
Another type of boundary Harnack principle, as studied by Kemper \cite{Kemper72}, concerns the H\"older regularity of the ratio $u_{1}/u_{2}$ above a Lipschitz graph $\Gamma$, where $u_{1}, u_{2} \geq 0$ satisfy  
\begin{equation}
    \Delta u_{i} = 0 \quad \text{in } \{x_{n} \geq \Gamma(x')\}, \quad u_{i} \Big|_{\Gamma} = 0.
\end{equation}
This type of boundary Harnack principle can be extended to operators $L$ with measurable coefficients (see \cite{CFMS, FGMS}) and to more general domains such as H\"older domains (see \cite{BB91, BB94, JK82}).\par
In this paper, due to the natural scaling of \eqref{algebraic}, we start by analyzing the case where $\epsilon = 1$. We define  
\[
    w := \frac{\phi_{1}}{\Psi_{1}} - 1,
\]
which satisfies the following inequality:
\begin{equation}\label{degerate semiequation1}
    \mathrm{div}(\Psi^{2} \nabla w) \geq \Psi^{1-s} w.
\end{equation}
By studying the regularity theory of \eqref{degerate semiequation1}, we establish Theorem \ref{boundary Harnack with limit 1}. However, the main challenge in analyzing \eqref{degerate semiequation1} is that the weight $ \Psi^{2}$ is degenerate.\par
A standard approach to handling degenerate equations is to check whether the degenerate weight $\lambda(x)$ belongs to the $A_{2}$-Muckenhoupt class, which requires that for any ball $B$,
\begin{equation}
    \left(\frac{1}{|B|} \int_{B} \lambda \right) \cdot \left(\frac{1}{|B|} \int_{B} \lambda^{-1} \right) \leq C.
\end{equation}
The classical $H^{1}$ theory for uniformly elliptic PDEs extends to equations with $A_{2}$-Muckenhoupt weights; see \cite{Fabes2, Fabes3, Fabes1}. However, in our setting, the weight $\Psi^{2}$ does not belong to the $A_{2}$-Muckenhoupt class.\par 
There are two approaches to studying a non-$A_{2}$-Muckenhoupt weighted degenerate equation. One approach is to introduce a weighted Sobolev space $H_{\lambda}^{1}$, where the seminorm is defined as  
\[
    \|w\|_{H_{\lambda}^{1}}^{2} := \int_{\Omega} \lambda |\nabla w|^{2} \, dx.
\]  
 Subsequently, they perform analysis within this weighted space (see \cite{DP23a, DP23b} for reference). The other approach is to perturb the operator  
\[
    A_{\epsilon}^{ij} = A^{ij} + \epsilon \delta^{ij}
\]  
and then obtain regularity estimates independent of $\epsilon$ (see \cite{STV1, STV2}). Our study of the equation \eqref{degerate semiequation1} primarily follows the first approach.

\subsection{Notations}
Throughout this paper, we reserve certain characters for specific quantities according to the following conventions:

\begin{itemize}
    \item \textbf{$\mathbb{T}^2$:} The flat torus defined as $[-\frac{1}{2},\frac{1}{2}]^2/{\sim}$, where $(x_1,x_2) \sim (y_1,y_2)$ if there exists a pair $(m,n) \in \mathbb{Z}^2$ such that $(x_1,x_2) = (y_1 + m, y_2 + n)$.

    \item \textbf{$\mathbb{T}^{++}$:} The first quadrant of the flat torus, given by $\left\{0 < x_1 < \frac{1}{2}, 0 < x_2 < \frac{1}{2}\right\}$.

    \item \textbf{$\mathcal{I}^2$:} The standard cube in $\mathbb{R}^2$, defined as  
    \begin{equation}
        \mathcal{I}^2 = \left[-\frac{1}{2},\frac{1}{2}\right] \times \left[-\frac{1}{2},\frac{1}{2}\right].
    \end{equation}

    \item \textbf{$\mathbb{R}^{++}$:} The first quadrant in the whole space, given by $\left\{x_1 > 0, x_2 > 0\right\}$.

    \item \textbf{$\chi_A(x)$:} The characteristic function of a set $A$.

    \item \textbf{$\operatorname{sgn}(x)$:} The sign function, defined as  
    \begin{align*}
        \operatorname{sgn}(x) =
        \begin{cases}
            -1, & x < 0, \\
            1,  & x \geq 0.
        \end{cases}
    \end{align*}

    \item \textbf{$\mathcal{G}$:} The Green's function of the Laplacian on the torus. We recall some properties of the Green's function in Lemma \ref{Green} in the appendix.

    \item \textbf{$\Delta^{-1}$:} The inverse Laplacian, defined as  
    \[
        \Delta^{-1} f = f * \mathcal{G}.
    \]

    \item \textbf{$\psi_{0}$:} The stream function of the Bahouri-Chemin patch, given by  
    \[
        \psi_0 = \Delta^{-1} \big[-\operatorname{sgn}(x_{1})\operatorname{sgn}(x_{2})\big].
    \]

    \item \textbf{$\Psi_{\epsilon}$:} The barrier function constructed in Lemma 3.5 of \cite{elgindi2022regular}. (We recall Lemma 3.5 in the Appendix.) In particular, for $\epsilon=1$, we simply write $\Psi := \Psi_{1}$.  

    \item \textbf{$\sigma$:} A positive constant that first appears in Theorems \ref{main1} and \ref{boundary Harnack with limit 1}.

    \item \textbf{$\alpha_s$:} Defined as  
    \[
        \alpha_s = \frac{1-s}{1+s},
    \]
    which is related to the optimal convergence result in Theorem \ref{main2}.

    \item \textbf{$x$:} A vector whose Cartesian expression is $(x_1, x_2)$.

    \item \textbf{$|\cdot|$:} For notational convenience, we use $|\cdot|$ ambiguously to denote the absolute value of a scalar and the standard norm in the flat torus.

    \item \textbf{$\|\cdot\|$:} Similarly, we use $\|\cdot\|$ to represent either the Euclidean norm in $\mathbb{R}^2$ or the norm of a certain functional space, depending on the context.
    \item \textbf{Constants $c(\delta), \cdots$:} Throughout this paper, we use lowercase letters such as $c(\delta)$ to denote constants that may depend on the parameter $\delta$, particularly in Section 2.

\item \textbf{Constants $C,\cdots$:} For convenience, we use the uppercase letter such as $C$ to denote constants that are independent of $\delta$. Those values may vary from line to line.

\end{itemize}
\subsection{Organization of the Paper}

In Section \ref{section of degenerate equation harnack}, we prove Theorem \ref{boundary Harnack with limit 1} by analyzing the equation satisfied by
\[
w = \frac{u}{\Psi} - 1,
\]
as given in the equation \eqref{degerate semiequation1} (and later in \eqref{linear RHS}). In this section, we develop an \( H^1 \) theory for the equation \eqref{linear RHS} that includes establishing an \( L^\infty \) estimate and proving a weak Harnack principle.

In Section \ref{section of optimal convergence}, based on Theorem \ref{boundary Harnack with limit 1}, we first derive the asymptotic expansion of \( \phi_{\epsilon} \) near the origin. We then substitute this expansion into the Green function formulation to complete the proof of Theorem \ref{main1}. The proof of Theorem \ref{main2} is presented subsequently.

Finally, in the Appendix we recall several useful facts and provide an extension of the boundary Harnack principle.

\section{The boundary Harnack principle}\label{section of degenerate equation harnack}
In this section, we will prove Theorem \ref{boundary Harnack with limit 1} with $R=1$:
\begin{theorem}\label{boundary Harnack with limit 1, radius 1}
    Let $0<s<1$ and $u\geq \Psi$ satisfy
    \begin{equation}
    \begin{aligned}
         &\Delta u=-u^{-s}\mbox{ in }B_{1}\cap\mathbb{R}^{++},\\&
         u(x_1,x_2)=0,\text{if $x_{1}=0$ or $x_{2}=0$}.
    \end{aligned}
    \end{equation}
    Then there exist constants $C,\sigma>0$ depending only on $s$, such that for any $x\in B_{1/2}\cap\mathbb{R}^{++}$, we have
    \begin{equation}
     w:=\frac{u}{\Psi}-1\leq C|x|^{\sigma} \max_{B_{1}}u,
    \end{equation}
    where $\Psi=\Psi_{1}$ is defined in \eqref{homogeneous solution definition}.
\end{theorem}
\begin{remark}
    We note that the condition $u\geq \Psi$ can be replaced by $u\geq 0$, leading to similar estimates. The proof of this alternative is interesting and will be presented in the Appendix.
\end{remark}
\begin{proof}[Proof of Theorem \ref{boundary Harnack with limit 1}]
    Given Theorem \ref{boundary Harnack with limit 1, radius 1},
     Theorem \ref{boundary Harnack with limit 1} follows from a scaling argument by studying $u_R(x)=R^{-\frac{2}{1+s}}u(Rx)$.
\end{proof}

The remainder of this section is devoted to proving Theorem \ref{boundary Harnack with limit 1, radius 1}.
\subsection{The Equation Satisfied by the Ratio $\displaystyle w=\frac{u}{\Psi}-1$}
We now proceed to prove Theorem \ref{boundary Harnack with limit 1, radius 1}. The method is to study the ratio
\begin{equation}\label{ratio definition}
    w:=\frac{u}{\Psi}-1\geq 0.
\end{equation}
It satisfies the inequalities \eqref{linear RHS} and \eqref{0 RHS} below:
\begin{lemma}\label{weak subsolution}
    Under the setting of Theorem \ref{boundary Harnack with limit 1, radius 1}, we have
    \begin{equation}\label{linear RHS}
        \mathrm{div}(\Psi^{2}\nabla w)\geq\Psi^{1-s}w.
    \end{equation}
    Consequently, we get
    \begin{equation}\label{0 RHS}
        \mathrm{div}(\Psi^{2}\nabla w)\geq0.
    \end{equation}
\end{lemma}
\begin{proof}
    By direct computation, we obtain
    \[
    \Psi \Delta u = \mathrm{div}(\Psi^{2}\nabla w) + \Psi(1+w)\Delta\Psi.
    \]
    By substituting $\Delta u = -u^{-s}$ and $\Delta\Psi = -\Psi^{-s}$, we have
    \[
    \mathrm{div}(\Psi^{2}\nabla w) = \Psi^{1-s}\frac{(1+w)^{1+s} - 1}{(1+w)^{s}}
    = \Psi^{1-s} f(w)w,
    \]
    where
    \[
    f(w) := \frac{(1+w)^{1+s}-1}{(1+w)^{s}w} \geq 1 \quad \text{for any } w \geq -1.
    \]
\end{proof}

We also note that the function $w$ defined in \eqref{ratio definition} is $L^2$ integrable.

\begin{lemma}[$L^{2}$ finiteness]\label{L2 estimate}
    Suppose $u \in L^{\infty}(B_{R} \cap \mathbb{R}^{++})$, $u \geq \Psi$, and
    \begin{equation}
        \Delta u = -u^{-s}, \quad u\big|_{\partial \mathbb{R}^{++}} = 0,
    \end{equation}
    then $u/\Psi \in L^{2}(B_{1/2} \cap \mathbb{R}^{++})$.
\end{lemma}
\begin{proof}
    Since $u \geq \Psi$, we have $|\Delta u| \leq \Psi^{-s}$. In particular, $\Delta u \in L^{p - \delta}(B_{1} \cap \mathbb{R}^{++})$ for any small $\delta > 0$, where $p = \frac{1 + s}{s}$. 

    By the Schauder estimates, we obtain $u \in C^{1}$. Then as $u = 0$ on the coordinate axes, we conclude that $u(x)=O(\min\{x_{1},x_{2}\})$. Using the asymptotic behavior of $\Psi$ from Lemma~\ref{Barrier function}, we conclude that
    \begin{equation}
        \int_{B_{1/2} \cap \mathbb{R}^{++}} \left( \frac{u}{\Psi} \right)^{2} dx < \infty.
    \end{equation}
\end{proof}

We now adapt the definition of weak solutions from Chapter 8 of \cite{GT77}.

\begin{definition}[Weak solution, adapted from Chapter 8 of \cite{GT77}]\label{def. weak solution}
We say that a function $w \in H^{1}_{\mathrm{loc}}(B_{1} \cap \mathbb{R}^{++}) \cap L^{2}(B_{1} \cap \mathbb{R}^{++})$ is a weak solution to
\begin{equation}
    \mathrm{div}(\Psi^2 \nabla w) \geq f,
\end{equation}
if \begin{equation}
    \int_{B_{1} \cap \mathbb{R}^{++}} \Psi^{2} \nabla w \cdot \nabla g \, dx 
    \leq -\int_{B_{1} \cap \mathbb{R}^{++}} f g \, dx.
\end{equation} for every test function $g \in H^{1}(\mathbb{R}^{++})$ with $g \geq 0$ and compact support in $B_{1}$, in the sense that $g$ vanishes on the set $\partial (\mathbb{R}^{++}\cap B_1)\setminus(\{x_1=0\}\cup \{x_2=0\})$(See the figure \ref{fig:support g distribution} for an illustration for the support of $g$).
\begin{center}\label{fig:support g distribution}
    \includegraphics[width=0.6\linewidth]{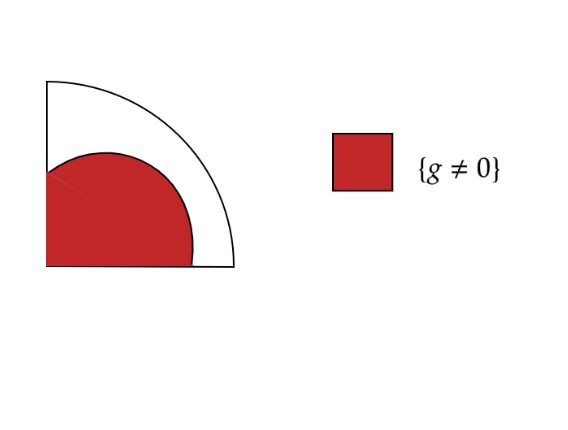}
     \captionof{figure}{illustrative figure for the support of $g$}
\end{center}
  \emph{Note that we do not require $g$ to vanish on the coordinate axes.}
\end{definition}

Next, we derive the $L^{\infty}$ estimate for the ratio $w$, which is a weak solution to \eqref{linear RHS} and \eqref{0 RHS}.

\subsection{$L^{\infty}$ Estimate of the Ratio}

In this subsection, we establish an interior $L^{\infty}$ estimate for the function $w$ defined in \eqref{ratio definition}. We start from stating the main result concerning the $L^{\infty}$ bound.

\begin{proposition}[$L^{\infty}$ estimate]\label{L infinity estimate}
    Assume that $w \in H^{1}_{\text{loc}}(B_{1} \cap \mathbb{R}^{++})$ is a weak solution of \eqref{0 RHS}(See Definition \ref{def. weak solution}). If $w \geq 0$ and 
    \[
    \int_{B_{1} \cap \mathbb{R}^{++}} w^{2} \, dx < \infty,
    \]
    then
    \[
    \|w\|_{L^{\infty}(B_{\frac{1}{2}} \cap \mathbb{R}^{++})}^{2} \leq C \int_{B_{1} \cap \mathbb{R}^{++}} \Psi^{2} w^{2} \, dx.
    \]
\end{proposition}

To prove Proposition \ref{L infinity estimate}, we employ the De Giorgi - Nash - Moser iteration method. We first establish a version of the Caccioppoli inequality and a suitable Sobolev inequality.
\begin{lemma}[Caccioppoli inequality]\label{Caccioppoli}
    Assume that $w \geq 0$ is a weak solution of \eqref{0 RHS} in $B_{1} \cap \mathbb{R}^{++}$, and that
    \[
        \int_{B_{1} \cap \mathbb{R}^{++}} w^{2} \, dx < \infty.
    \]
    Let $\eta \in C^{\infty}_c(B_{1})$ be a cutoff function with $0 \leq \eta \leq 1$, which doesn't have to vanish on the set $\{x_1=0\}\cup \{x_2=0\}$ (Similar to the test function $g$ in Definition \ref{def. weak solution}). Then
    \begin{equation}
        \int_{B_{1} \cap \mathbb{R}^{++}} \Psi^{2} |\nabla(\eta w)|^{2} \, dx
        \leq C \int_{B_{1} \cap \mathbb{R}^{++}} \Psi^{2} w^{2} |\nabla \eta|^{2} \, dx.
    \end{equation}
\end{lemma}
\begin{proof}
    Let $\eta_{h} = \eta \cdot \iota_{h}$, where $\iota_{h}$ is another smooth cutoff function satisfying
\begin{equation}
    \iota_{h}\big|_{\{\Psi \geq h\}} = 1, \quad 
    \iota_{h}\big|_{\partial \mathbb{R}^{++} \cap B_{1}} = 0, \quad 
    |\nabla \iota_{h}| \leq \frac{C}{h}.
\end{equation}
Clearly, $\Psi^{2} |\nabla \iota_{h}|^{2} \leq C \chi_{\{\Psi \leq h\}}$. 

By multiplying both sides of the inequality $\mathrm{div}(\Psi^{2} \nabla w) \geq 0$ by $\eta_{h}^{2} w$ and integrating over $B_{1} \cap \mathbb{R}^{++}$, we have
\begin{equation}
    \int_{B_{1} \cap \mathbb{R}^{++}} \Psi^{2} \left( |\nabla(\eta_{h} w)|^{2} - w^{2} |\nabla \eta_{h}|^{2} \right) \, dx 
    = \int_{B_{1} \cap \mathbb{R}^{++}} \Psi^{2} \nabla w \cdot \nabla(\eta_{h}^{2} w) \, dx \leq 0.
\end{equation}
Using the facts that $0 \leq \eta \leq 1$ and $\Psi^{2} |\nabla \iota_{h}|^{2} \leq C \chi_{\{\Psi \leq h\}}$, we obtain
\begin{equation}
    \int_{B_{1} \cap \mathbb{R}^{++}} \Psi^{2} |\nabla(\eta w)|^{2} \chi_{\{\Psi \geq h\}} \, dx 
    \leq C \int_{B_{1} \cap \mathbb{R}^{++}} \Psi^{2} w^{2} |\nabla \eta|^{2} \, dx 
    + C \int_{B_{1} \cap \mathbb{R}^{++}} w^{2} \chi_{\{\Psi \leq h\}} \, dx.
\end{equation}
The desired inequality follows by sending $h \to 0$ and applying the Fatou's lemma.
\end{proof}
\begin{lemma}[Weighted Sobolev inequality]\label{Sobolev}
    Assume that $w \in H^{1}_{\mathrm{loc}}(\mathbb{R}^{++})$ and $w \to 0$ as $r \to \infty$. Then
    \begin{equation}\label{H^1}
        C \int_{\mathbb{R}^{++}} \Psi^{2} |\nabla w|^{2} \, dx 
        \geq \frac{\displaystyle \int_{\mathbb{R}^{++}} \Psi^{4} w^{4} \, dx}
                   {\displaystyle \int_{\mathbb{R}^{++}} \Psi^{2} w^{2} \, dx}
        \geq \frac{\displaystyle \int_{\mathbb{R}^{++}} \Psi^{2} w^{2} \, dx}
                   {\displaystyle \int_{\mathbb{R}^{++}} \chi_{\{w \neq 0\}} \, dx}.
    \end{equation}
\end{lemma}
\begin{proof}
    We first claim that under the same assumptions, the following weighted $W^{1,1}$ Sobolev inequality holds:
    \begin{equation}\label{W^1,1}
        C \int_{\mathbb{R}^{++}} \Psi^{2} |\nabla w| \, dx 
        \geq \left( \int_{\mathbb{R}^{++}} \Psi^{4} w^{2} \, dx \right)^{1/2}.
    \end{equation}
    
    Applying this inequality to $w^2$ and using the Cauchy-Schwarz inequality, we obtain
    \begin{equation}
        \int_{\mathbb{R}^{++}} \Psi^{4} w^{4} \, dx 
        \leq C \left( \int_{\mathbb{R}^{++}} \Psi^{2} w |\nabla w| \, dx \right)^{2}
        \leq C \left( \int_{\mathbb{R}^{++}} \Psi^{2} w^{2} \, dx \right)
               \left( \int_{\mathbb{R}^{++}} \Psi^{2} |\nabla w|^{2} \, dx \right).
    \end{equation}
\end{proof}
 The proof of the $W^{1,1}$ Sobolev inequality \eqref{W^1,1} is provided below.
\begin{proof}[Proof of inequality \eqref{W^1,1}]
    Based on the asymptotic behavior of $\Psi$ from Lemma~\ref{Barrier function}, we first introduce an equivalent weight
    \[
        \widetilde{\Psi} = r^{\frac{2}{s+1}} \sin(2\theta),
    \]
    and prove a version of \eqref{W^1,1} with $\Psi$ replaced by $\widetilde{\Psi}$. We then use polar coordinates $(r, \theta)$ with $r \geq 0$ and $0 \leq \theta \leq \frac{\pi}{2}$ to parametrize $\mathbb{R}^{++}$, and define the rays
    \[
        L_{\theta} = \left\{ x \in \mathbb{R}^{++} : \frac{x_2}{x_1} = \tan \theta \right\}
    \]
    (see Figure~\ref{sobolev figure}, line \textcircled{1}). We now correspondingly define the quantities
    \[
        E = \int_{\mathbb{R}^{++}} \widetilde{\Psi}^{2} |\nabla w| \, dx_1 dx_2, \quad 
        E(\theta) = \int_{L_{\theta}} r \widetilde{\Psi}^{2} |\nabla w| \, dr, \quad 
        I(\theta) = \int_{L_{\theta}} r \widetilde{\Psi}^{4} w^{2} \, dr.
    \]
    Since $dx_1 dx_2 = r \, dr \, d\theta$, we have
    \[
        \int_{0}^{\pi/2} E(\theta) \, d\theta = E.
    \]
    Therefore, it suffices to prove for all $\theta \in [0, \pi/2],$ we have 
    \begin{equation}\label{eq. sobolev theta to integrate}
        \sqrt{2} \, E(\theta) E \geq I(\theta).
    \end{equation}

    By integrating along the ray $L_{\theta}$ (see Figure~\ref{sobolev figure}, line \textcircled{1}), we have
    \begin{equation}\label{estimate sobolev 1}
        r \widetilde{\Psi}^{2} |w(r, \theta)| 
        \leq r \widetilde{\Psi}^{2}(r, \theta) \int_{r}^{\infty} |\partial_r w(\rho, \theta)| \, d\rho 
        \leq \int_{r}^{\infty} \rho \widetilde{\Psi}^{2}(\rho, \theta) |\partial_r w(\rho, \theta)| \, d\rho 
        \leq E(\theta).
    \end{equation}

    On the other hand, since $\widetilde{\Psi}$ is monotonically increasing in both $x_1$ and $x_2$, by integrating vertically (line \textcircled{2} in Figure~\ref{sobolev figure}), we get
    \begin{equation}\label{estimate sobolev 2}
        \widetilde{\Psi}^{2}(x_1, x_2) |w(x_1, x_2)| 
        \leq \widetilde{\Psi}^{2}(x_1, x_2) \int_{x_2}^{\infty} |\partial_2 w(x_1, h)| \, dh 
        \leq \int_{x_2}^{\infty} \widetilde{\Psi}^{2}(x_1, h) |\partial_2 w(x_1, h)| \, dh.
    \end{equation}
    For $(x_1, x_2) \in L_{\theta}$ with $0 \leq \theta \leq \frac{\pi}{4}$, we integrate \eqref{estimate sobolev 2} and get
    \begin{equation}
        \int_{L_{\theta}} \widetilde{\Psi}^{2} |w| \, dr 
        \leq \frac{1}{\cos \theta} \int_{L_{\theta}} \widetilde{\Psi}^{2} |w| \, dx_1 
        \leq \sqrt{2} E.
    \end{equation}
    Then by symmetry, we have \begin{equation}\label{estimate sobolev 3}
        \int_{L_{\theta}} \widetilde{\Psi}^{2} |w| \, dr 
        \leq\sqrt{2} \, E, \text{for all $\theta\in [0,\frac{\pi}{2}]$.}
    \end{equation}

    Combining \eqref{estimate sobolev 1} and \eqref{estimate sobolev 3}, we obtain
    \begin{equation}
        I(\theta) 
        \leq \left( \sup_{L_{\theta}} r \widetilde{\Psi}^{2} |w(r, \theta)| \right)
              \int_{L_{\theta}} \widetilde{\Psi}^{2} |w| \, dr 
        \leq \sqrt{2} \, E(\theta) E,
    \end{equation}
    which completes the proof of \eqref{eq. sobolev theta to integrate}.
    
    \begin{center}
        \label{sobolev figure}
        \includegraphics[width=0.6\linewidth]{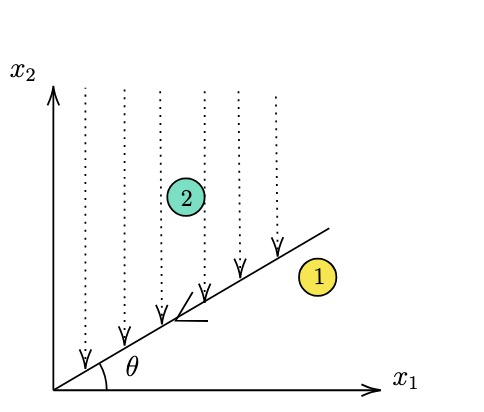}
        \captionof{figure}{Integration along lines \textcircled{1} and \textcircled{2} yields estimates \eqref{estimate sobolev 1} and \eqref{estimate sobolev 2}, respectively.}
    \end{center}
\end{proof}

\begin{proof}[Proof of Proposition \ref{L infinity estimate}]
    Assume 
    \[
    \int_{B_{1}\cap\mathbb{R}^{++}}\Psi^{2}w^{2}\leq\varepsilon_{0}
    \]
    for some sufficiently small $\varepsilon_{0}$. It suffices to show that $w\leq1$ in $B_{1/2}$. For $k\geq1$, define
    \[
    t_{k}=1-2^{-k},\quad r_{k}=\frac{1}{2}+2^{-k},
    \]
    and set
    \[
    w_{k}=(w-t_{k})_{+},\quad A_{k}=\int_{B_{r_{k}}\cap \mathbb{R}^{++}}\Psi^{2}w_{k}^{2}.
    \]
    Note that $A_{1},A_{2}\leq\varepsilon_{0}$, and our goal is to show that
    \[
    \lim_{k\to\infty}A_{k}=0,
    \]
    provided $\varepsilon_{0}$ is small.

    Let $\eta_{k}$ be a family of cutoff functions (to be inserted in Lemma \ref{Caccioppoli}) satisfying
    \[
    \operatorname{supp}(\eta_{k})=B_{r_{k}}\cap\mathbb{R}^{++},\quad
    \eta_{k}\big|_{B_{r_{k+1}}\cap\mathbb{R}^{++}}=1,\quad
    |\nabla\eta_{k}|\leq C2^{k}.
    \]
    Then, by applying the Caccioppoli inequality, we have
    \[
    \int_{B_{r_{k}}\cap\mathbb{R}^{++}}\Psi^{2}|\nabla(\eta_{k}w_{k})|^{2}\leq C4^{k}\int_{B_{r_{k}}\cap\mathbb{R}^{++}}\Psi^{2}w_{k}^{2}.
    \]
    Next, we apply the Sobolev inequality from Lemma \ref{Sobolev} to $\eta_{k}w_{k}$ and we get
    \[
    \int_{B_{r_{k+1}}\cap \mathbb{R}^{++}}\Psi^{2}w_{k}^{2}\leq C4^{k}\left(\int_{B_{r_{k}}\cap\mathbb{R}^{++}}\Psi^{2}w_{k}^{2}\right)
    \left(\int_{B_{r_{k}}\cap\mathbb{R}^{++}}\chi_{\{w_{k}>0\}}\right).
    \]
    Since $w_{k+1}\leq w_{k}$, we have
    \[
    A_{k+1}\leq\int_{B_{r_{k+1}}\cap \mathbb{R}^{++}}\Psi^{2}w_{k}^{2}.
    \]
    On the other hand, we note that 
    \[
    \{w_{k}>0\} = \{w_{k-1}>t_{k}-t_{k-1}=2^{-k}\}.
    \]
   As a consequence,
    \[
    A_{k-1}\geq\int_{B_{r_{k}}\cap \mathbb{R}^{++}}\Psi^{2}w_{k-1}^{2}\geq 4^{-k}\int_{B_{r_{k}}\cap \mathbb{R}^{++}}\Psi^{2}\chi_{\{w_{k}>0\}}.
    \]
    Since $B_{r_{k}}\cap \mathbb{R}^{++}\subseteq B_{2}\cap\mathbb{R}^{++}$, by applying H\"older's inequality, we get
    \begin{align*}
        \int_{B_{r_{k}}\cap\mathbb{R}^{++}}\chi_{\{w_{k}>0\}}
        &\leq \left(\int_{B_{r_{k}}\cap\mathbb{R}^{++}}\Psi^{2}\chi_{\{w_{k}>0\}}\right)^{\frac{1}{p}}
        \left(\int_{B_{2}\cap\mathbb{R}^{++}}\Psi^{-\frac{2}{p-1}}\right)^{\frac{p-1}{p}} \\
        &\leq C\left(\int_{B_{r_{k}}\cap\mathbb{R}^{++}}\Psi^{2}\chi_{\{w_{k}>0\}}\right)^{1/p}
        \leq C\left(4^{k}A_{k-1}\right)^{1/p},
    \end{align*}
    for a sufficiently large $p$. By combining these estimates, we have  the following inductive inequality
    \[
    A_{k+1}\leq C\,16^{k}A_{k}A_{k-1}^{1/p}.
    \]
    Taking logarithms on both sides, we can show if $A_1,A_2$ are small, we have $A_{k}\to0$ as $k\to\infty$ .
\end{proof}

\subsection{Weak Harnack Principle of the Ratio}
In this section we introduce and prove a weak Harnack principle.

\begin{proposition}[Weak Harnack Principle]\label{weak harnack}
    Assume that $w\in H^{1}_{loc}$ is defined in the circular sector $B_{1}\cap\mathbb{R}^{++}$ with $w\leq 1$, and satisfies 
    \[
    \int_{B_{1}\cap\mathbb{R}^{++}}w^{2}\,dx <\infty.
    \]
    If $w$ is a weak solution to \eqref{0 RHS} and $w$ satisfies 
    \[
    \int_{B_{1}\cap\mathbb{R}^{++}}\Psi^{2}\chi_{\{w\leq0\}}\,dx\geq\delta,
    \]
    then 
    \begin{equation}\label{eq. weak harnack result}
    w\leq 1-c(\delta) \quad \text{in } B_{1/2}\cap\mathbb{R}^{++},
    \end{equation}
    where $c(\delta)$ is a positive constant depending only on $\delta$.
\end{proposition}

To prove Proposition \ref{weak harnack}, we first state a version of isoperimetric inequality:

\begin{lemma}[Bi-partition Isoperimetric Inequality]\label{isoperimetric inequality}
    For every $\delta>0$, there exists $\sigma(\delta)>0$ such that for every function $w\in H^{1}_{loc}$ satisfying
    \begin{equation}\label{bi-partition condition}
    \int_{B_{1}\cap\mathbb{R}^{++}}\Psi^{2}\chi_{\{w\leq0\}}\,dx\geq\delta\int_{B_{1}\cap\mathbb{R}^{++}}\Psi^{2}\,dx \quad \text{and} \quad \int_{B_{1}\cap\mathbb{R}^{++}}\Psi^{2}\chi_{\{w\geq1\}}\,dx\geq\delta\int_{B_{1}\cap\mathbb{R}^{++}}\Psi^{2}\,dx,
    \end{equation}
   we have the following inequality:
    \begin{equation}\label{isoperimetric inequality estimate}
    \int_{B_{1}\cap\mathbb{R}^{++}}\Psi^{2}|\nabla w|\,dx\geq\sigma(\delta).
    \end{equation}
\end{lemma}

\begin{proof}
    Let $h = h(\delta) > 0$ be sufficiently small so that
    \begin{equation}\label{eq:small-h}
        \int_{B_{1}\cap\mathbb{R}^{++}}\Psi^{2}\chi_{\{\Psi < h(\delta)\}}\,dx
        \leq \frac{\delta}{2}\int_{B_{1}\cap\mathbb{R}^{++}}\Psi^{2}\,dx.
    \end{equation}
    Then, in the region 
    \[
    \Omega := \{\Psi \geq h(\delta)\}\cap B_{1}\cap\mathbb{R}^{++},
    \]
    it follows that
    \begin{equation}\label{eq:omega-measure}
        \int_{\Omega}\Psi^{2}\chi_{\{w\leq0\}}\,dx
        \geq \frac{\delta}{2}\int_{B_{1}\cap\mathbb{R}^{++}}\Psi^{2}\,dx.
    \end{equation}
     From the equation \eqref{eq:omega-measure},
    \begin{equation}\label{eq:chi-bound}
        \int_{\Omega}\chi_{\{w\leq0\}}\,dx
       \geq \frac{1}{\|\Psi\|^2_{L^{\infty}(B_{1}\cap\mathbb{R}^{++})}}
        \int_{\Omega}\Psi^{2}\chi_{\{w\leq0\}}\,dx\geq \frac{\delta}{2\|\Psi\|^2_{L^{\infty}(B_{1}\cap\mathbb{R}^{++})}}\int_{B_1\cap \mathbb{R}^{++}}\Psi^2dx:=c_{1}(\delta).
    \end{equation}
    Similarly, we also obtain
    \begin{equation}\label{eq. w>=1 large}
    \int_{\Omega}\chi_{\{w\geq1\}}\,dx \geq c_{1}(\delta).
    \end{equation}
    Let us  denote
    \begin{equation}
        \mathop{avg}_{\Omega}w = \frac{1}{|\Omega|}\int_{\Omega}w(x)\,dx.
    \end{equation}
    If $\ds\mathop{avg}_{\Omega}w\leq\frac{1}{2}$, we apply \eqref{eq. w>=1 large} and conclude
    \begin{equation}
        \int_{\Omega}\Bigl|w - \mathop{avg}_{\Omega}w\Bigr|\,dx
        \geq\frac{1}{2}\int_{\Omega}\chi_{\{w\geq1\}}dx\geq\frac{c_{1}(\delta)}{2}.
    \end{equation}
    By symmetry, we have the following lower bound estimate:
    \begin{equation}
        \int_{\Omega}\Bigl|w - \mathop{avg}_{\Omega}w\Bigr|\,dx
        \geq\frac{c_{1}(\delta)}{2}.
    \end{equation}
    We then apply the standard Poincar\'e inequality to $w$ in $\Omega$, and obtain
    \begin{equation}\label{eq:poincare}
        \int_{\Omega}|\nabla w|\,dx\geq \frac{1}{C}\int_{\Omega}\Bigl|w - \mathop{avg}_{\Omega}w\Bigr|\,dx \geq \frac{c_1({\delta})}{2C}.
    \end{equation}
    Finally, since $\Psi \geq h(\delta)$ in $\Omega$, we have
    \begin{equation}\label{eq:final}
        \int_{B_{1}\cap\mathbb{R}^{++}}\Psi^{2}|\nabla w|\,dx 
        \geq h(\delta)^{2}\int_{\Omega}|\nabla w|\,dx 
        \geq \frac{c_1({\delta})h(\delta)^{2}}{2C}
        =: \sigma(\delta).
    \end{equation}
\end{proof}

Now we are able to prove Proposition \ref{weak harnack}.

\begin{proof}[Proof of Proposition \ref{weak harnack}]
    For $k\geq 1$, we first define 
    \[
    t_{k} = 1 - 2^{-k} \quad\text{and}\quad
    w_{k} = 2\Biggl(\frac{w - t_{k}}{1-t_{k}}\Biggr)_{+}, \quad 
    \overline{w}_{k} = \min\{w_{k},1\}.
    \]
    Let us denote
    \[
    r(\delta) = \frac{\delta}{\pi\|\Psi\|_{L^{\infty}(B_{1})}^{2}},
    \]
    then
    \begin{align}\label{eq:nonzero-measure}
    \int_{B_{1-r(\delta)}\cap\mathbb{R}^{++}}\Psi^{2}\chi_{\{\overline{w}_{k}\leq 0\}}\,dx \geq&\int_{B_{1-r(\delta)}\cap\mathbb{R}^{++}}\Psi^{2}\chi_{\{w\leq 0\}}\,dx-\int_{(B_{1}\setminus B_{1-r(\delta)})\cap\mathbb{R}^{++}}\Psi^{2}dx\\
    \geq&\delta-\|\Psi\|_{L^{\infty}(B_{1})}^{2}\cdot\Big|(B_{1}\setminus B_{1-r(\delta)})\cap\mathbb{R}^{++}\Big|\geq\frac{\delta}{2}.
    \end{align}
    By choosing the cutoff function $\eta(x)\geq0$ supported in $B_{1}\cap\mathbb{R}^{++}$ satisfying
    \begin{equation}
        \eta=1\mbox{ in }B_{1-r(\delta)}\cap\mathbb{R}^{++},\quad|\nabla\eta|\leq2r(\delta)^{-1},
    \end{equation}
in Lemma~\ref{Caccioppoli} and using the fact $0\leq w_{k}\leq 2$, there exists some uniform $C_{1}$, such that
    \begin{align}\label{eq:grad-bound}
    &\int_{B_{1-r(\delta)}\cap\mathbb{R}^{++}}\Psi^{2} |\nabla \overline{w}_{k}|^{2}\,dx 
    \leq\int_{B_{1-r(\delta)}\cap\mathbb{R}^{++}}\Psi^{2} |\nabla w_{k}|^{2}\,dx\\
    \leq&\int_{B_{1}\cap\mathbb{R}^{++}}\Psi^{2} |\nabla(\eta w_{k})|^{2}\,dx
    \leq\int_{B_{1}\cap\mathbb{R}^{++}}\Psi^{2}w_{k}^{2}|\nabla\eta|^{2}dx\\
    \leq&\frac{16}{r(\delta)^{2}}\int_{B_{1}\cap\mathbb{R}^{++}}\Psi^{2}dx\leq C_{1}r(\delta)^{-2}.
    \end{align}
     Note that the truncated functions $w_{k}$ are all weak solutions to \eqref{0 RHS}, by Proposition \ref{L infinity estimate}, there exists a uniformly small $C_{2}$, such that if $\int_{B_{1-r(\delta)}\cap\mathbb{R}^{++}}\Psi^{2}\chi_{\{\overline{w}_{k}\geq 0\}}\,dx\leq C_{2}$, we have
    \begin{equation}
        w_{k}\leq\frac{1}{10}\mbox{ \, in }B_{1/2}\cap\mathbb{R}^{++}.
    \end{equation}
    Let us denote
    \[
    K_{0} = 100\left\lceil \frac{C_{1}r(\delta)^{-2}}{\sigma(\min\{\frac{\delta}{2},C_{2}\})^{2}} \int_{B_{1}\cap\mathbb{R}^{++}}\Psi^{2}\,dx \right\rceil + 100,
    \]
    where $\sigma(\cdot)$ is the function obtained in Lemma~\ref{isoperimetric inequality}. We now claim that
    \begin{equation}\label{eq. claim in weak harnack}
        \int_{B_{1-r(\delta)}\cap\mathbb{R}^{++}}\Psi^{2}\chi_{\{\overline{w}_{k}\geq 0\}}\,dx\leq C_{2},\mbox{ for all }k\geq K_{0}.
    \end{equation}
Based on the above claim, by choosing $c(\delta)=2^{-K_{0}-100}$ in the desired estimate \eqref{eq. weak harnack result}, we have 
$w_{K_{0}}\leq\frac{1}{10}$ in $B_{1/2}\cap\mathbb{R}^{++}$. This implies \begin{equation}
        w\leq1-2^{-K_{0}}+\frac{2^{-K_{0}}}{20}\leq1-c(\delta)\mbox{ in }B_{1/2}\cap\mathbb{R}^{++},
    \end{equation}and we finish the proof.\\ Now we are going to prove the claim \eqref{eq. claim in weak harnack}.
    Suppose on the contrary, that the claim is not true. Then, by Lemma \ref{isoperimetric inequality}, we have 
    \begin{equation}\label{eq:grad-lower}
    \int_{B_{1-r(\delta)}\cap\mathbb{R}^{++}}\Psi^{2} |\nabla \overline{w}_{k}|\,dx \ge \sigma\Big(\min\{\frac{\delta}{2},C_{2}\}\Big), \text{for every $k\leq K_{0}$.}
    \end{equation}
    This, together with the bound in \eqref{eq:grad-bound} and the Cauchy-Schwartz inequality, implies that
    \begin{equation}\label{eq:level-set-vol}
    \int_{B_{1-r(\delta)}\cap\mathbb{R}^{++}} \Psi^{2}\chi_{\{t_{k} < w < t_{k+1}\}}\,dx 
    \ge \frac{\sigma(\min\{\frac{\delta}{2},C_{2}\})^{2}}{C_{1}r(\delta)^{-2}}.
    \end{equation}
    Since the regions $\{t_{k} < w < t_{k+1}\}$ are pairwise disjoint, by summing the inequality \eqref{eq:level-set-vol} over $k=1,2,\dots,K_{0}$, we find that
    \begin{equation}
        \int_{B_{1-r(\delta)}\cap\mathbb{R}^{++}} \Psi^{2}dx\geq\sum_{k=1}^{K_{0}}\int_{B_{1-\delta_{1}}\cap\mathbb{R}^{++}}\Psi^{2}\chi_{\{t_{k} < w < t_{k+1}\}}\,dx>100\int_{B_{1}\cap\mathbb{R}^{++}} \Psi^{2}dx,
    \end{equation}
    which is a contradiction. Then, we prove the claim \eqref{eq. claim in weak harnack}.
\end{proof}

\subsection{Proof of Theorem \ref{boundary Harnack with limit 1, radius 1}}
We now prove Theorem \ref{boundary Harnack with limit 1, radius 1} by studying the ratio
\[
w = \frac{u}{\Psi} - 1.
\]
In the proof, we make use of the following consequence of the ABP estimate.

\begin{lemma}\label{ABP}
    Assume that \(A\in C^{1}\) is a $2\times2$ symmetric matrix satisfying \(\lambda I\leq A\leq\Lambda I\) for some positive constants \(\lambda\) and \(\Lambda\). Suppose that \(u\) satisfies
    \[
    \mathrm{div}(A\nabla u)\geq\chi_{S} \quad \text{in } B_{1},
    \]
    where \(S\subset B_{1}\) is a measurable set with 
    \[
    |B_{1}\setminus S| \leq \delta,
    \]
    and where \(\delta>0\) is a small number depending on the matrix \(A\). Then there exists \(\sigma>0\), depending only on \(A\), such that 
    \[
    \max_{B_{1/2}} u \leq \max_{B_{1}} u - \sigma.
    \]
\end{lemma}
By Lemma~\ref{ABP}, the positivity of the right-hand side of \eqref{linear RHS} implies the following lemma.
\begin{lemma}\label{uniform squezzing}
    Assume $0\leq w\leq1$ is a weak solution of \eqref{linear RHS} in $B_{1}\cap\mathbb{R}^{++}$. Then there exist $\sigma,\delta>0$, such that
    \begin{equation}
        \int_{B_{1}\cap\mathbb{R}^{++}}\Psi^{2}\chi_{\{w\leq1-\sigma\}}dx\geq\delta.
    \end{equation}
\end{lemma}

\begin{proof}
    If there exists some constant $\delta$ independent of the choice of $w$, such that
    \begin{equation}
        \int_{B_{1}\cap\mathbb{R}^{++}}\Psi^{2}\chi_{\{w\leq 1/2\}}\,dx \geq \delta,
    \end{equation}
     we have already proven the Lemma (in this case, $\sigma=\frac{1}{2}$). Otherwise, suppose that for some sufficiently small $\delta$,
    \begin{equation}\label{w mostly large revised}
        \int_{B_{1}\cap\mathbb{R}^{++}}\Psi^{2}\chi_{\{w\leq 1/2\}}\,dx \leq \delta.
    \end{equation}
   Now we consider the interior disk 
    \[
    \mathcal{D} = \bigl\{ (x_{1}-\tfrac{1}{2})^{2} + (x_{2}-\tfrac{1}{2})^{2} \leq \tfrac{1}{100} \bigr\},
    \]
    as illustrated in the figure below.
     \begin{center}\label{fig:ABP picture}
        \includegraphics[width=0.6\linewidth]{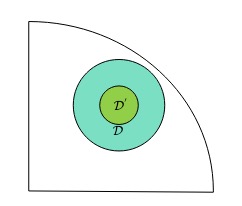}
\captionof{figure}{Illustrative figure for the domain $\mathcal{D}$}
    \end{center}
    In the set 
    \[
    S = \mathcal{D}\cap \{ w\geq 1/2 \},
    \]
    the right-hand side of \eqref{linear RHS} satisfies
    \[
    \Psi^{1-s}w \geq C,
    \]
    for a uniformly small constant \(C>0\). By \eqref{w mostly large revised} and that $\Psi$ has a positive lower bound in $\mathcal{D}$, when \(\delta\) is sufficiently small, the set \(\mathcal{D}\setminus S\) has small measure. Consequently, by applying Lemma~\ref{ABP} to \eqref{linear RHS} in the disk \(\mathcal{D}\), there exists some \(\sigma>0\) such that
    \begin{equation}\label{ABP conclusion}
        w \leq \max_{x\in \mathcal{D}} w(x) - \sigma \leq 1 - \sigma 
    \end{equation}
    in the smaller disk
    \[
    \mathcal{D}' = \Bigl\{ (x_{1}-\tfrac{1}{2})^{2} + (x_{2}-\tfrac{1}{2})^{2} \leq \tfrac{1}{400} \Bigr\}.
    \]
    Then, as $\delta_{1}$ is sufficiently small, we have
    \begin{equation}
        \int_{B_{1}\cap\mathbb{R}^{++}}\Psi^{2}\chi_{\{w\leq1-\sigma\}}dx\geq\int_{\mathcal{D}'}\Psi^{2}dx\geq\delta,
    \end{equation}
    thus proving the lemma.
\end{proof}

Now we give the proof of Theorem \ref{boundary Harnack with limit 1, radius 1}. The idea is to incorporate Lemma~\ref{uniform squezzing} with the weak Harnack principle (Proposition~\ref{weak harnack}).

\begin{proof}[Proof of Theorem \ref{boundary Harnack with limit 1, radius 1}]
    By Lemma \ref{L2 estimate} and Proposition \ref{L infinity estimate}, we have
    \[
    \|w\|_{L^{\infty}(B_{1/2}\cap\mathbb{R}^{++})} \leq C \max_{B_{1}\cap\mathbb{R}^{++}} u.
    \]
   We now denote
    \[
    M := \max_{B_{1/2}\cap\mathbb{R}^{++}} w.
    \]Then, we apply Proposition \ref{weak harnack} to the function 
    \[
    v(x) := w(x) - \bigl(1-\sigma_{1}\bigr) M,
    \]
    where \(\sigma_{1}\) is the constant from Lemma~\ref{uniform squezzing}. Consequently,  we obtain
    \[
    \max_{B_{1/4}\cap\mathbb{R}^{++}} w \leq (1-\sigma_{2}) M,
    \]
    for some \(\sigma_{2} < \sigma_{1}\).
    
   We notice that if $w(x)$ satisfies equation \eqref{linear RHS}, then as $\Psi$ has the homogeneous degree $\frac{2}{1+s}$, $w_{r}(x)=w(rx)$ also satisfies \eqref{linear RHS}. Then, by iterating the oscillation reduction argument above on balls $B_{2^{-k}}\cap\mathbb{R}^{++}$, we deduce that there exist constants \(C > 0\) and \(\sigma>0\) (with \(\sigma \leq \log_{1/2}(1-\sigma_{2})\)) such that
    \[
    w(x) \leq C\,|x|^{\sigma} M \leq C\,|x|^{\sigma} \max_{B_{1}\cap\mathbb{R}^{++}} u.
    \]
    This completes the proof of Theorem \ref{boundary Harnack with limit 1, radius 1}.
\end{proof}

\section{Sharp Regularity and the Optimal Convergence of $\phi_{\epsilon}$}\label{section of optimal convergence}
In this section, we prove the continuity of $\phi_{\epsilon}$ with respect to $\epsilon$ in $C^{1,\alpha_s}(\mathbb{T}^2):=C^{1,\frac{1-s}{1+s}}(\mathbb{T}^2)$ topology, and establish the sharp convergence rate of $\phi_{\epsilon}\to\psi_0$ in the same space as $\epsilon\to0$.

\subsection{Proof of Theorem \ref{phi with upper bound}}\label{boundary harnack subsection} 
We now prove Theorem \ref{phi with upper bound}. The key observation is that $\phi_{\epsilon}$ and $\Psi_{\epsilon}$ satisfy the same PDE near the origin, as shown by the following localization lemma:

\begin{lemma}[Domain localization]\label{domain}
For sufficiently small $\epsilon > 0$, the following estimate holds in the rescaled neighborhood:
\begin{equation}
\phi_{\epsilon}(x) \leq \epsilon \quad \text{for all } x \in \mathbb{T}^{++} \cap B_{R(\epsilon)}, \quad \text{where } R(\epsilon) \coloneqq \frac{\sqrt{\epsilon}}{-\ln \epsilon}.
\end{equation}
\end{lemma}
\begin{proof}
We begin with the decomposition:
\begin{equation}\label{eq:decomposition}
    \phi_{\epsilon} = \left(\phi_{\epsilon} - \frac{\epsilon}{2}\right)_+ + \min\left\{\phi_{\epsilon}, \frac{\epsilon}{2}\right\} \leq \left(\phi_{\epsilon} - \frac{\epsilon}{2}\right)_+ + \frac{\epsilon}{2}.
\end{equation}
Based on the above decomposition, it suffices to establish the following bound for the truncated function $\phi_{\mathrm{cut}} \coloneqq (\phi_{\epsilon} - \epsilon/2)_+$:
\begin{equation}\label{eq:cut_estimate}
    \phi_{\mathrm{cut}}(x) \leq \frac{\epsilon}{2} \quad \text{for}\quad |x| \leq R(\epsilon) \coloneqq \frac{\sqrt{\epsilon}}{-\ln \epsilon}.
\end{equation}

We notice that $\phi_{\mathrm{cut}}$ satisfies the following differential inequality in the weak sense:
\begin{equation}\label{eq:diff_ineq}
    \Delta \phi_{\mathrm{cut}} \geq 2^s \Delta \psi_0 \text{\, in $[0,\frac{1}{2}]^2,$}
\end{equation}
with both $\phi_{\mathrm{cut}}$ and $\psi_0$ vanishing on $\partial([0,\frac{1}{2}]^2)$. It then follows from the maximum principle that
\begin{equation}\label{eq:max_principle}
    \phi_{\mathrm{cut}}(x) \leq 2^s \psi_0(x) \quad \text{in}\quad [0,\tfrac{1}{2}]^2.
\end{equation}
The required estimate \eqref{eq:cut_estimate} then follows from the asymptotic behavior of $\psi_0$ near the origin as in Lemma~\ref{BC patch}.
\end{proof}

\begin{proof}[Proof of Theorem \ref{phi with upper bound}]
The lower bound $\phi_{\epsilon} \geq \Psi_{\epsilon}$ follows immediately from Theorem \ref{Barrier previous}. For the upper bound, we proceed as follows.

By Lemma \ref{domain}, we have the pointwise estimate:
\begin{equation}\label{eq:pointwise_bound}
    \phi_{\epsilon}(x) \leq \epsilon \quad \text{for} \quad x \in \mathbb{T}^{++} \cap B_{R(\epsilon)}, \quad R(\epsilon) \coloneqq \frac{\sqrt{\epsilon}}{-\ln \epsilon}.
\end{equation}

We now consider the rescaled function:
\begin{equation}\label{eq:rescaled_phi}
    \Phi_{\epsilon}(y) \coloneqq \frac{1}{\epsilon}\phi_{\epsilon}(\sqrt{\epsilon}y), \quad y = \frac{x}{\sqrt{\epsilon}} \in B_{\rho(\epsilon)} \cap \mathbb{R}^{++}, \quad \rho(\epsilon) \coloneqq \frac{1}{-\ln \epsilon}.
\end{equation}
This scaling ensures that $0 \leq \Phi_{\epsilon}(y) \leq 1$ and $\Delta_y \Phi_{\epsilon} = -\Phi_{\epsilon}^{-s}$ in $B_{\rho(\epsilon)} \cap \mathbb{R}^{++}$.

By Theorem \ref{boundary Harnack with limit 1}, we obtain the following ratio estimate:
\begin{equation}\label{eq:harnack_ratio}
    \frac{\Phi_{\epsilon}(y)}{\Psi(y)} - 1 \leq C\left|\frac{y}{\rho(\epsilon)}\right|^{\sigma} \rho(\epsilon)^{-\frac{2}{1+s}} \|\Phi_{\epsilon}\|_{L^\infty(B_{\rho(\epsilon)} \cap \mathbb{R}^{++})} \leq C\left|\frac{y}{\rho(\epsilon)}\right|^{\sigma} \rho(\epsilon)^{-\frac{2}{1+s}}.
\end{equation}

Now we rescale back to original variables  in the sense that $x = \sqrt{\epsilon}y$, $R(\epsilon) = \sqrt{\epsilon}\rho(\epsilon)$).  We then  conclude that for $|x| \leq R(\epsilon)$ and $\epsilon \leq \epsilon_0$,
\begin{equation}\label{eq:final_bound}
    \frac{\phi_{\epsilon}(x)}{\Psi_{\epsilon}(x)} - 1 \leq C\left|\frac{x}{R(\epsilon)}\right|^{\sigma} \rho(\epsilon)^{-\frac{2}{1+s}} = C\epsilon^{-\sigma/2}|x|^{\sigma}|\ln \epsilon|^{\sigma - \frac{2}{1+s}}\leq\epsilon^{-\sigma}|x|^{\sigma}.
\end{equation}
\end{proof}

\subsection{Proof of Theorem \ref{main1}}
We present the proof of Theorem \ref{main1} in this section by combining Lemma \ref{GreensFunction} with Theorem \ref{phi with upper bound}. The key strategy involves analyzing the difference function $\phi_\epsilon - \Psi_\epsilon$, whose Laplacian $\Delta(\phi_\epsilon - \Psi_\epsilon)$ exhibits better integrability properties compared to $\Delta\phi_\epsilon$ alone.
\begin{proof}[Proof of Theorem \ref{main1}]
 We first denote $\mathcal{I}^2$ as the standard cube in $\mathbb{R}^2$: \begin{equation}
\mathcal{I}^2=[-\frac{1}{2},\frac{1}{2}]\times [-\frac{1}{2},\frac{1}{2}],
\end{equation}
and define \begin{equation}
    R_{\epsilon}(y):=G_{\epsilon}(\phi_{\epsilon})(y)-\frac{\epsilon^{s}}{\Psi_{\epsilon}^{s}}+sgn(y_1)sgn(y_2).
\end{equation}
    Then, we have the following decomposition: \begin{equation}\label{division}
    \begin{aligned}
        &\quad\nabla \big(\phi_{\epsilon}-\Psi_{\epsilon}-\psi_0\big)(\Tilde{x})-\nabla \big(\phi_{\epsilon}-\Psi_{\epsilon}-\psi_0\big)(x)\\&= \int_{\mathbb{T}^2}\bigg(\big(\nabla \mathcal{G}(\Tilde{x},y)-\frac{(\Tilde{x}-y)}{|\Tilde{x}-y|^2}\big)-\big(\nabla \mathcal{G}(x,y)-\frac{(x-y)}{| x-y| ^2}\big)\bigg)[G_{\epsilon}(\phi_{\epsilon})(y)+sgn(y_1)sgn(y_2)]dy\\&+\big(\int_{\Tilde{x}+(\mathcal{I}^2)^{c}}\frac{\Tilde{x}-y}{\|\Tilde{x}-y\|^2}\frac{\epsilon^{s}}{\Psi_{\epsilon}^{s}}dy-\int_{x+(\mathcal{I}^2)^{c}}\frac{x-y}{\|x-y\|^2}\frac{\epsilon^{s}}{\Psi_{\epsilon}^{s}}dy\big)
        +\int_{B_{2|x-\Tilde{x}|(x)}^{c}\cap \mathbb{T}^2}\big(\frac{\Tilde{x}-y}{| \Tilde{x}-y| ^2}-\frac{x-y}{| x-y| ^2}\big) R_{\epsilon}(y)dy\\&+\int_{B_{2|x-\Tilde{x}|(x)}}\frac{\Tilde{x}-y}{|\Tilde{x}-y|^2}R_{\epsilon}(y)dy+\int_{B_{2|x-\Tilde{x}|(x)}}\frac{x-y}{|x-y|^2}R_{\epsilon}(y)dy\\&:=E_1+E_2+I_1+I_2+I_3,\text{  here $\|\cdot\|$ represent the Euclidean norm in $\mathbb{R}^2$.}
    \end{aligned}
    \end{equation}
   \begin{center}
       \includegraphics[width=0.8\linewidth]{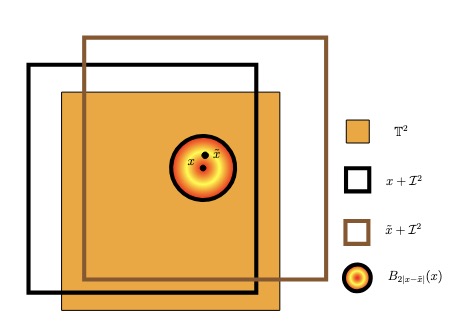}
\captionof{figure}{Illustrative figure for the domain division in estimating \eqref{division}}
   \end{center}
  Using Lemma \ref{GreensFunction}, we obtain 
    \begin{equation}
        \begin{aligned}
            |E_1|&\leq C|x-\Tilde{x}|\int_{\mathbb{T}^2} |G_{\epsilon}(\phi_{\epsilon})+sgn(y_1)sgn(y_2)|dy\\&
           \leq C|x-\Tilde{x}|\int_{\{\psi_0<\epsilon\}\cap Q}\frac{\epsilon^s}{\Psi_{\epsilon}^{s}}dy\leq C|x-\Tilde{x}|\int_{\{\psi_0<\epsilon\}\cap Q}\frac{\epsilon^s}{\psi_{0}^{s}}dy=O(\epsilon^{s}|x-\Tilde{x}|)
        \end{aligned}
    \end{equation}
    \par
\textbf{Estimate of $E_2$}\\
Let us denote $S_{x}=x+(\mathcal{I}^{2})^{c}$ and $S_{\Tilde{x}}=\Tilde{x}+(\mathcal{I}^{2})^{c}$. 
From our assumptions, for all $y \in S_{x}\cap S_{\Tilde{x}}$, by fundamental theorem of calculus, we have
\begin{equation}
\left\|\frac{x-y}{\|x-y\|^2}-\frac{\Tilde{x}-y}{\|\Tilde{x}-y\|^2}\right\| \leq \frac{C\|x-\Tilde{x}\|}{\|y\|^2}. 
\end{equation} 
Furthermore,  we have 
\begin{equation}
\begin{aligned}
|E_2|&\leq \left|\int_{S_{x}\cap S_{\Tilde{x}}} \left(\frac{x-y}{\|x-y\|^2}-\frac{\Tilde{x}-y}{\|\Tilde{x}-y\|^2}\right)\frac{\epsilon^{s}}{\Psi_{\epsilon}^{s}}dy\right| \\
&\quad + \int_{S_{x}\Delta S_{\Tilde{x}}} \left(\frac{1}{\|x-y\|}+\frac{1}{\|\Tilde{x}-y\|}\right)\frac{\epsilon^{s}}{|\Psi_{\epsilon}^{s}|}dy \\
&:= E_{21}+E_{22}
\end{aligned} 
\end{equation}

First, we have
\begin{equation}
E_{21}\leq C\|x-\Tilde{x}\|\int _{S_{x}\cap S_{\Tilde{x}}}\frac{\epsilon^{\frac{s}{s+1}}}{r^{\frac{2+4s}{s+1}}\sin^{s}{(2\theta)}}dy_{1}dy_{2}=O(\epsilon^{\frac{s}{s+1}}|x-\Tilde{x}|).
\end{equation}
In addition, as $|S_{x}\Delta S_{\Tilde{x}}|\leq C|x-\Tilde{x}|$, we end up with 
\begin{equation}
E_{22} \leq C \int_{S_{x}\Delta S_{\Tilde{x}}} \frac{\epsilon^{s}}{|\Psi_{\epsilon}|^{s}}dy\leq C \int_{S_{x}\Delta S_{\Tilde{x}}} \frac{\epsilon^{\frac{s}{s+1}}}{r^{\frac{2s}{s+1}}|\sin^{s}(2\theta)|}dy\leq C\epsilon^{\frac{s}{s+1}}|x-\Tilde{x}|.
\end{equation}
Therefore, $|E_2|=O(\epsilon^{\frac{s}{s+1}}|x-\Tilde{x}|)$.

\textbf{Estimate of $I_1$}\\
   Let us define the domain $\Omega:=B_{2|x-\Tilde{x}|(x)}^{c}\cap \{(y_{1},y_{2})\big|\,\frac{1}{2}>y_{1}>y_{2}>0\}$. By the fundamental theorem of calculus and various symmetry assumptions, to obtain an upper bound for $I_1$, it suffices to control
    \begin{equation}
        |x-\Tilde{x}|\int_{\Omega}\frac{1}{|x-y|^2}|R_{\epsilon}(\phi_{\epsilon})(y)|dy.
    \end{equation}

We first divide $\Omega$ into two regions:
\begin{equation}
    \Omega_{\epsilon} := \Omega \cap \left\{y \mid |y| \leq \dfrac{\epsilon^{1/2}}{\ln(1/\epsilon)}\right\}\mbox{ as }\Omega_{\epsilon}^{c} := \Omega \setminus \Omega_{\epsilon}.
\end{equation}
This leads to the decomposition:
\begin{equation}
    \begin{aligned}
         &\quad\big|x-\Tilde{x}\big|\int_{\Omega}\frac{\big|R_{\epsilon}(\phi_{\epsilon})(y)\big|}{|x-y|^2}dy\\
        &\leq \big|x-\Tilde{x}\big|\big(\int_{\Omega_{\epsilon}}\frac{|R_{\epsilon}(\phi_{\epsilon})(y)+1|}{|x-y|^2}dy\big)+\big|x-\Tilde{x}\big|\int_{\Omega_{\epsilon}}\frac{1}{|x-y|^2}dy\\& +|x-\Tilde{x}|\int_{\Omega_{\epsilon }^{c}}\frac{1}{|x-y|^2}\frac{\epsilon^{s}}{|\Psi_{\epsilon}|^{s}}dy+|x-\Tilde{x}|\int_{\Omega_{\epsilon}^{c}}\frac{|G_{\epsilon}(\phi_{\epsilon})(y)-1|}{|x-y|^2}dy\\
        &:=J_1+J_2+J_3+J_4.
        \end{aligned}
        \end{equation}
See the figure below for the domain division in calculating $I_1$.
\begin{center}
    \includegraphics[width=0.6\linewidth]{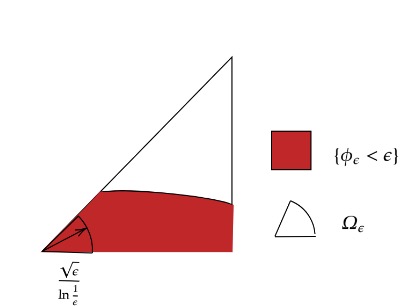}
\captionof{figure}{Illustrative figure for the domain division for calculating $I_1$.}
\end{center}
Estimate of $J_1$: 
    By Theorem \ref{phi with upper bound} and H\"older's inequality with exponent $\frac{1}{p}=\frac{2s}{s+1}-\frac{\sigma}{2}$, we obtain
    \begin{equation}
    \begin{aligned}
        J_1&\leq \epsilon^{\frac{s}{s+1}-\sigma}|x-\Tilde{x}| \int_{\Omega_{\epsilon}}\frac{1}{|x-y|^{\frac{2}{s+1}+\frac{\sigma}{2}}}(\frac{1}{|y_2|^{\frac{2s}{1+s}-\sigma}} \frac{1}{|x-y|^{\frac{2s}{s+1}-\frac{\sigma}{2}}})dy
        \\&\leq \epsilon^{\frac{s}{s+1}-\sigma}|x-\Tilde{x}|\bigg | \int_{\Omega_{\epsilon}^{c}}\frac{1}{|x-y|^{\frac{2p^{*}}{s+1}+\frac{p^{*}\sigma}{2}}}dy\bigg|^{\frac{1}{p*}}\bigg|\int_{\Omega_{\epsilon}^{c}}\big(\frac{1}{|y_2|^{\frac{2s}{1+s}-\sigma}}\frac{1}{|x-y|^{\frac{2s}{s+1}-\frac{\sigma}{2}}}\big)^{p}dy\bigg|^{\frac{1}{p}}
        \\& \leq  C \epsilon^{\frac{s}{s+1}-\sigma}|x-\Tilde{x}|^{\frac{1-s}{1+s}+\frac{\sigma}{2}}.
        \end{aligned}
    \end{equation}
Estimate for $J_2$: By applying  H\"older's inequality with $\frac{1}{p}=\frac{1}{1+s}+\frac{\sigma}{4}$, we have \begin{equation}
    J_2\leq |x-\Tilde{x}|\bigg(\int_{\Omega_{\epsilon}}\frac{1}{|x-y|^{2p}}dy\bigg)^{\frac{1}{p}}\big|\Omega_{\epsilon}\big|^{\frac{1}{p*}}\leq \epsilon^{\frac{s}{1+s}-2\sigma}\big|x-\Tilde{x}\big|^{\frac{1-s}{1+s}+\frac{\sigma}{2}}.
\end{equation}
Estimate for $J_3$:
Using H\"older's inequality with $\frac{1}{p}=s+\frac{s(1-s)}{s+1}-\frac{\sigma}{2}$, we get
\begin{equation}
\begin{aligned}
     J_3=&\epsilon^{\frac{s}{4}}|x-\Tilde{x}| \int_{\Omega_{\epsilon}^{c}} \frac{1}{|x-y|^{2-s-\delta}}\big(\frac{1}{|x-y|^{s+\delta}}\frac{1}{|y_{2}|^{s}}\big)dy\\
     &\leq \epsilon^{\frac{s}{4}}|x-\Tilde{x}| \bigg(\int_{\Omega_{\epsilon}^{c}} \frac{1}{|x-y|^{(2-s-\delta)p*}}dy\bigg)^{\frac{1}{p*}}\bigg(\int_{\Omega_{\epsilon}^{c}}[\frac{1}{|x-y|^{(s+\delta)p}}\frac{1}{|y_{2}|^{sp}}]dy\bigg)^{\frac{1}{p}}\\&\leq C \epsilon^{\frac{s}{4}}|x-\Tilde{x}|^{\frac{1-s}{1+s}+\frac{\sigma}{2}}.
\end{aligned}
\end{equation}
Estimate for $J_4$: By H\"older's inequality with $\frac{1}{p}=\frac{1}{1+s}+\frac{\sigma}{4}$, we get 
\begin{equation}
\begin{aligned}
    J_4&\leq \epsilon^{s}|x-\Tilde{x}|\int_{\{\psi_0<\epsilon\}\cap \Omega}\frac{1}{|x-y|^2}\frac{1}{|\psi_0(y)|^s}dy\\&\leq\epsilon^{s}|x-\Tilde{x}|\bigg(\int_{\{\psi_0<\epsilon\}\cap \Omega} \frac{1}{|x-y|^{2p}}dy\bigg)^{\frac{1}{p}}\bigg(\int_{\{\psi_0<\epsilon\}\cap \Omega}\frac{1}{|\psi_{0}|^{p^{*}s}(y)}dy\bigg)^{\frac{1}{p^{*}}}\\
    &\leq\epsilon^{s}|x-\Tilde{x}|^{\frac{1-s}{1+s}+\frac{\sigma}{2}}.
\end{aligned}
\end{equation}
Combining these estimates, we end up with
 \begin{equation}
    |I_1|\leq C \epsilon^{\frac{s}{4}}|x-\Tilde{x}|^{\frac{1-s}{1+s}+\frac{\sigma}{2}}.
\end{equation}

\textbf{Estimate of $I_2$}:\par
We define $\Tilde{\Omega}:=B_{2|x-\Tilde{x}|}(x)\cap \left\{0<y_{2}<y_{1}<\frac{1}{2}\right\}$. By symmetry assumptions, to obtain a upper bound for $I_2$, it suffices to control $\int_{\Tilde{\Omega}}\frac{|R_{\epsilon}(y)|}{|\Tilde{x}-y|} dy$.
Partition $\widetilde{\Omega}$ into
\begin{equation}
    \widetilde{\Omega}_\epsilon := \widetilde{\Omega} \cap \left\{y \mid |y| \leq \epsilon^{1/2}/\ln(1/\epsilon)\right\}\mbox{ and }\widetilde{\Omega}_\epsilon^c := \widetilde{\Omega} \setminus \widetilde{\Omega}_\epsilon.
\end{equation}
This gives the decomposition:
\begin{equation}
    \begin{aligned}
       &\int_{\Tilde{\Omega}}\frac{|R_{\epsilon}(y)|}{|\Tilde{x}-y|} dy  =\int_{\Tilde{\Omega}_{\epsilon}}\frac{|R_{\epsilon}(y)+1|}{|x-y|}dy+\int_{\Tilde{\Omega}_{\epsilon}}\frac{1}{|x-y|}dy+\int_{\Tilde{\Omega}_{\epsilon}^{c}}\frac{1}{|x-y|}\frac{\epsilon^{s}}{|\psi_0|^s}dy+\int_{\Tilde{\Omega}_{\epsilon}^{c}}\frac{|G_{\epsilon}(\phi_{\epsilon})(y)-1|}{|x-y|}dy\\&:=M_1+M_2+M_3+M_4.
    \end{aligned}
\end{equation}
See the figure below for the domain division in estimating $I_2$.
\begin{center}
    \includegraphics[width=0.6\linewidth]{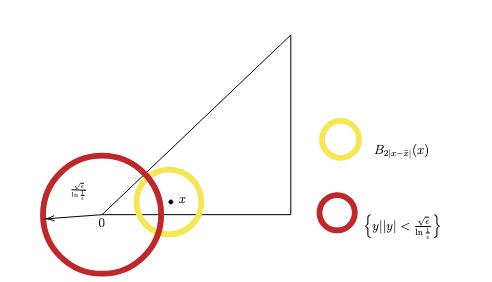}
\captionof{figure}{Illustrative figure for the domain division for calculating $I_2$.}
\end{center}
Estimate of $M_1$:
Again by Theorem \ref{phi with upper bound} and H\"older inequality with  $\frac{1}{p}=\frac{2s}{s+1}-\frac{\sigma}{2}$, we have 
\begin{equation}
    \begin{aligned}
         M_1&\leq \epsilon^{\frac{s}{s+1}-\sigma}\int_{\Tilde{\Omega}_{\epsilon}}  \frac{1}{|x-y|^{\frac{1-s}{s+1}+\frac{\sigma}{2}}}\big( \frac{1}{|x-y|^{\frac{2s}{s+1}-\frac{\sigma}{2}}}\frac{1}{|y_{2}|^{\frac{2s}{s+1}-\sigma}}\big)dy\\&\leq \epsilon^{\frac{s}{s+1}-\sigma}\bigg(\int_{\Tilde{\Omega}_{\epsilon}}\big(\frac{1}{|x-y|^{\frac{1-s}{s+1}+\frac{\sigma}{2}}}\big)^{p^{*}}dy\bigg)^{\frac{1}{p^{*}}}\bigg(\int_{\Tilde{\Omega}_{\epsilon}} \big(\frac{1}{|x-y|^{\frac{2s}{s+1}-\frac{\sigma}{2}}} \frac{1}{|y_{2}|^{\frac{2s}{s+1}-\sigma}}\big)^{p}dy\bigg)^{\frac{1}{p}}\\&\leq C \epsilon^{\frac{s}{s+1}-\sigma}|x-\Tilde{x}|^{\frac{1-s}{1+s}+\frac{\sigma}{2}}.  
    \end{aligned}
\end{equation}
 Estimate for $M_2$: Using H\"older inequality with $\frac{1}{p}=\frac{1}{1+s}+\frac{\sigma}{4}$,  we get \begin{equation}
    M_2\leq \bigg(\int_{\Tilde{\Omega}_{\epsilon}}\frac{1}{|x-y|^{p}}dy\bigg)^{\frac{1}{p}}\big|\Omega_{\epsilon}\big|^{\frac{1}{p*}}\leq C\epsilon^{\frac{s}{1+s}-2\sigma}\big|x-\Tilde{x}\big|^{\frac{1-s}{1+s}+\frac{\sigma}{2}}.
\end{equation}
Estimate for $M_3$:
By H\"older inequality with $\frac{1}{p}=s+\frac{s(1-s)}{s+1}-\frac{\sigma}{2}$, we have 
\begin{equation}
\begin{aligned}
     M_3=&\epsilon^{\frac{s}{4}} \int_{\Tilde{\Omega}_{\epsilon}^{c}} \frac{1}{|x-y|^{1-s-\delta}}\big(\frac{1}{|x-y|^{s+\delta}}\frac{1}{|y_{2}|^{s}}\big)dy\\
     &\leq \epsilon^{\frac{s}{4}} \bigg(\int_{\Tilde{\Omega}_{\epsilon}^{c}} \frac{1}{|x-y|^{(1-s-\delta)p*}}dy\bigg)^{\frac{1}{p*}}\bigg(\int_{\Tilde{\Omega}_{\epsilon}^{c}}[\frac{1}{|x-y|^{(s+\delta)p}}\frac{1}{|y_{2}|^{sp}}]dy\bigg)^{\frac{1}{p}}\\&\leq C\epsilon^{\frac{s}{4}}|x-\Tilde{x}|^{\frac{1-s}{1+s}+\frac{\sigma}{2}}.
\end{aligned}
\end{equation}
Estimate for $M_4$: 
Applying H\"older inequality with $\frac{1}{p}=\frac{1}{1+s}+\frac{\sigma}{4}$, we have 
\begin{equation}
\begin{aligned}
    M_4&\leq \epsilon^{s}\int_{\{\psi_0<\epsilon\}\cap \Omega}\frac{1}{|x-y|}\frac{1}{|\psi_0(y)|^s}dy\\&\leq\epsilon^{s}\bigg(\int_{\{\psi_0<\epsilon\}\cap \Omega} \frac{1}{|x-y|^{p}}dy\bigg)^{\frac{1}{p}}\bigg(\int_{\{\psi_0<\epsilon\}\cap \Omega}\frac{1}{|\psi_{0}|^{p^{*}s}(y)}dy\bigg)^{\frac{1}{p^{*}}}\\&\leq C\epsilon^{s}|x-\Tilde{x}|^{\frac{1-s}{1+s}+\frac{\sigma}{2}}.
\end{aligned}
\end{equation}
As a result, we have the following. \begin{equation}
    |I_2|\leq \epsilon^{\frac{\epsilon}{4}}|x-\Tilde{x}|^{\frac{1-s}{s+1}+\frac{\sigma}{2}}.
\end{equation}
Similarly, we have \begin{equation}
    |I_3|\leq \epsilon^{\frac{\epsilon}{4}}|x-\Tilde{x}|^{\frac{1-s}{s+1}+\frac{\sigma}{2}}.
\end{equation}
Now we have proved \begin{equation}
    \begin{aligned}
         &\quad|\nabla \big(\phi_{\epsilon}-\Psi_{\epsilon}-\psi_0\big)(\Tilde{x})-\nabla \big(\phi_{\epsilon}-\Psi_{\epsilon}-\psi_0\big)(x)|\leq C \epsilon^{\frac{s}{s+1}}|x-\Tilde{x}|+\epsilon^{\frac{s}{4}}|x-\Tilde{x}|^{\frac{1-s}{s+1}+\frac{\sigma}{2}}.
    \end{aligned}
\end{equation}
By setting $\sigma_{0}=\frac{\sigma}{2}$, we now obtain the desired estimate in Theorem \ref{main1}.
\end{proof}
\subsection{Proof of Theorem \ref{main2}}
In this section, We are going to complete the proof of Theorem \ref{main2}, establishing both the continuity of $\phi_\epsilon$ in $\epsilon$ and its optimal convergence to $\psi_0$.

\begin{proof}[Proof of Theorem \ref{main2}]
\textbf{Part 1: Continuity of $\phi_\epsilon$ in $\epsilon$}.  
Assume by contradiction that $\phi_\epsilon$ is not continuous at some $\epsilon_0 > 0$ in the $C^{1,\alpha_s}(\mathbb{T}^2) \coloneqq C^{1,\frac{1-s}{1+s}}(\mathbb{T}^2)$ topology. By the symmetry of $\phi_\epsilon$, there exist $\delta > 0$ and a sequence $\epsilon_n \to \epsilon_0$ such that
\begin{equation}\label{eq:discontinuity}
    \|\phi_{\epsilon_0} - \phi_{\epsilon_n}\|_{C^{1,\alpha_s}(\mathbb{T}^2 \cap B_{1/3}(0))} \geq \frac{1}{10} \|\phi_{\epsilon_0} - \phi_{\epsilon_n}\|_{C^{1,\alpha_s}(\mathbb{T}^2)} > \frac{\delta}{10}.
\end{equation}

Due to the decomposition $\phi_\epsilon = \Psi_\epsilon + (\phi_\epsilon - \Psi_\epsilon)$ and the continuity of $\Psi_\epsilon$ in $\epsilon$, we may assume
\begin{equation}\label{eq:diff_estimate}
    \| (\phi_{\epsilon_n} - \Psi_{\epsilon_n}) - (\phi_{\epsilon_0} - \Psi_{\epsilon_0}) \|_{C^{1,\alpha_s}(\mathbb{T}^2 \cap B_{1/3}(0))} > \frac{\delta}{20}.
\end{equation}

As in \eqref{Higher Holder Estimates1}, $\phi_\epsilon - \Psi_\epsilon$ is uniformly bounded in $C^{1,\alpha_s + \sigma}(\mathbb{T}^2 \cap B_{1/3}(0))$. Thus, along a subsequence $\{\phi_{\epsilon_n^{'}}\}$,  there exists  $\widetilde{\phi}_{\epsilon_0}$ in $C^{1,\alpha_s}$, 
$\phi_{\epsilon_n^{'}}$ converges to $\widetilde{\phi}_{\epsilon_0}$ in $C^{1,\alpha_s}(\mathbb{T}^2)$.
Moreover, $\widetilde{\phi}_{\epsilon_0}$ satisfies the equation
\begin{equation}\label{eq:limit_estimate}
    \|\widetilde{\phi}_{\epsilon_0} - \phi_{\epsilon_0}\|_{C^{1,\alpha_s}(\mathbb{T}^2)} \geq \|\widetilde{\phi}_{\epsilon_0} - \phi_{\epsilon_0}\|_{C^{1,\alpha_s}(\mathbb{T}^2 \cap B_{1/3}(0))}\geq\frac{\delta}{20}.
\end{equation}

However, since $\phi_{\epsilon_n} = \Delta^{-1}(G_{\epsilon_n}(\phi_{\epsilon_n}))$, the limit $\widetilde{\phi}_{\epsilon_0}$ satisfies:
\begin{itemize}
    \item $\Delta \widetilde{\phi}_{\epsilon_0} = G_{\epsilon_0}(\widetilde{\phi}_{\epsilon_0})$ in $\mathbb{T}^2$
    \item $\widetilde{\phi}_{\epsilon_0}$ is odd-odd symmetric
    \item $\widetilde{\phi}_{\epsilon_0}(x) \geq 0$ for $x \in \mathbb{T}^{++}$
\end{itemize}

By the maximum principle in \cite{Pino1992AGE}, $\phi_{\epsilon_0}$ is the unique non-negative solution to \eqref{algebraic} in $\mathbb{T}^{++}$, yielding $\widetilde{\phi}_{\epsilon_0} = \phi_{\epsilon_0}$, which contradicts \eqref{eq:limit_estimate}.

\medskip
\textbf{Part 2: Optimal Convergence to $\psi_0$}. We note first 
Lemma \ref{Barrier function} shows that
\begin{equation}\label{eq:psi_regularity}
    \lim_{\epsilon \to 0} \|\Psi_\epsilon\|_{C^{1,\alpha_s}(\mathbb{T}^2)} = 0 \quad \text{and} \quad \Psi_\epsilon \notin C^{1,\beta} \text{ for any } \beta > \alpha_s.
\end{equation}

Due the decomposition $\phi_\epsilon = \Psi_\epsilon + (\phi_\epsilon - \Psi_\epsilon)$,  we apply Theorem \ref{main1} and obtain
\[
\lim_{\epsilon \to 0} \|\phi_\epsilon - \psi_0\|_{C^{1,\alpha_s}(\mathbb{T}^2 \cap B_{1/3}(0))} = 0.
\]
We can now use the odd-odd symmetry and periodicity of $\phi_\epsilon$ and $\psi_0$ to extend the above convergence result to all points of $\mathbb{T}^2$:
\[
\lim_{\epsilon \to 0} \|\phi_\epsilon - \psi_0\|_{C^{1,\alpha_s}(\mathbb{T}^2)} = 0.
\]

For any $\beta \in (\alpha_s, \alpha_s + \sigma)$, we observe that:
\begin{itemize}
    \item $(\phi_\epsilon - \widetilde{\psi}_\epsilon) \in C^{1,\beta}$ near the origin
    \item $\Psi_\epsilon \notin C^{1,\beta}$ by \eqref{eq:psi_regularity}
\end{itemize}
Thus, we have shown $\phi_\epsilon \notin C^{1,\beta}(\mathbb{T}^2)$, proving the optimality of the $\alpha_s$-regularity.
\end{proof}

\appendix

\section{Useful Technical Results}
We collect the key technical results used in proving our main theorems.

\begin{lemma}[Green's Function on $\mathbb{T}^2$]\label{GreensFunction}
The Green's function for the Laplacian on $\mathbb{T}^2$ admits the decomposition:
\begin{equation}\label{Green}
\mathcal{G}(x,y) = \frac{1}{2\pi}\ln{|x-y|} + f(x-y),
\end{equation}
where $f \in C^\infty(\mathbb{T}^2 \times \mathbb{T}^2)$ is a smooth function. Here, $|x-y|$ denotes the periodic distance on $\mathbb{T}^2$. See \cite{Denisov2012DoubleEG} for proof.
\end{lemma}

\begin{lemma}[Bahouri-Chemin Patch]\label{BC patch}
Let $\psi_0$ be the stream function for the Bahouri-Chemin patch. Near the origin, in polar coordinates $(r,\theta)$ it satisfies:
\begin{equation}
\psi_0(r,\theta) \sim -r^2 \ln r \cdot \sin(2\theta).
\end{equation}
For complete details, see \cite{Bahouri1994EquationsDT}.
\end{lemma}

\begin{lemma}[Barrier Function]\label{Barrier function}
There exists a solution $\Psi_{\epsilon}$ to the boundary value problem:
\begin{equation}\label{special}
\begin{cases}
\Delta \Psi = -\dfrac{\epsilon^s}{\tilde \psi^{s}} & \text{in } \mathbb{R}^{++}, \\
\Psi = 0 & \text{on } \partial\mathbb{R}^{++}.
\end{cases}
\end{equation}
with the following characteristics:
\begin{itemize}
\item Regularity: $\Psi_{\epsilon} \in C^{1,\alpha_s}_{\text{loc}}(\mathbb{R}^{2})$ for $\alpha_{s}=\frac{1-s}{1+s}$
\item Growth rate estimate in polar coordinate:
\begin{equation}\label{maxiun principle 18}
\frac{1}{C}\epsilon^{\frac{s}{s+1}}r^{\frac{2}{s+1}}\sin(2\theta) < \Psi_{\epsilon}(r,\theta) < C\epsilon^{\frac{s}{s+1}}r^{\frac{2}{s+1}}\sin(2\theta)
\end{equation}
\end{itemize}
The proof of Lemma \ref{Barrier function} is contained in \cite{elgindi2022regular}.
\end{lemma}

\section{Generalization of Theorem \ref{boundary Harnack with limit 1, radius 1}}
We now extend Theorem \ref{boundary Harnack with limit 1, radius 1} by relaxing the lower bound assumption from $u \geq \Psi$ to $u \geq 0$. The generalized result is as follows:

\begin{theorem}\label{thm:generalized_BHI}
    Let $0 < s < 1$ and suppose $u \geq 0$ satisfies:
    \begin{equation}\label{eq:general_PDE}
    \begin{cases}
        \Delta u = -u^{-s} & \text{in } B_1 \cap \mathbb{R}^{++}, \\
        u(x_1,x_2) = 0 & \text{if } x_1 = 0 \text{ or } x_2 = 0.
    \end{cases}
    \end{equation}
    Then there exist constants $C, \sigma > 0$ (depending only on $s$) such that for all $x \in B_{1/2} \cap \mathbb{R}^{++}$:
    \begin{equation}\label{eq:general_estimate}
        -C|x|^\sigma \leq \frac{u(x)}{\Psi(x)} - 1 \leq C|x|^\sigma \max_{B_1} u,
    \end{equation}
    where $\Psi = \Psi_1$ is the homogeneous solution defined in \eqref{homogeneous solution definition}.
\end{theorem}

\begin{proof}
    Let $w = u/\Psi - 1$ as in \eqref{ratio definition}, and decompose it into positive and negative parts:
    \begin{equation}\label{eq:w_decomposition}
        w^+ = \max\{w, 0\}, \quad w^- = \max\{-w, 0\}.
    \end{equation}
    
    \noindent\textbf{Step 1: $L^2$ estimates.}  
    By construction, $0 \leq w^- \leq 1$, so $w^-$ is automatically $L^2$-integrable. For $w^+$, we construct a barrier function $v$ satisfying:
    \begin{equation}\label{eq:barrier}
    \begin{cases}
        \Delta v = -v^{-s} & \text{in } B_1 \cap \mathbb{R}^{++}, \\
        v|_{\partial(B_1 \cap \mathbb{R}^{++})} = \max\{u, \Psi\}.
    \end{cases}
    \end{equation}
    Since $v \geq (1 + w^+)\Psi$, Lemma \ref{L2 estimate} yields:
    \begin{equation}\label{eq:L2_bound}
        \|w^+\|_{L^2(B_{3/4} \cap \mathbb{R}^{++})} \leq \|v/\Psi\|_{L^2(B_{3/4} \cap \mathbb{R}^{++})} < \infty.
    \end{equation}
    Moreover, $w^{\pm}$ are both weak solutions to equations \eqref{linear RHS} and \eqref{0 RHS}.
    
    \noindent\textbf{Step 2: Convergence analysis.}  
    Applying the proof methodology of Theorem \ref{boundary Harnack with limit 1, radius 1} separately to $w^+$ and $w^-$:
\begin{enumerate}
    \item For $w^+ \geq 0$: We follow every single step in the proof of Theorem~\ref{boundary Harnack with limit 1, radius 1}, obtaining the upper bound in \eqref{eq:general_estimate} with the $\max_{B_1} u$ factor.
    
    \item For $w^- \geq 0$: The $L^{\infty}$ estimate (Proposition~\ref{L infinity estimate}) is omitted since $w^- \leq 1$ holds by construction, and the remaining steps proceed identically to establish the lower bound.
\end{enumerate}
The combination of these two cases yields the two-sided estimate \eqref{eq:general_estimate}.
\end{proof}

\section*{Acknowledgement}
The authors would like to thank Columbia University for organizing the conference on free-boundary problems to bring the authors together. The guidance from Ovidiu Savin about the boundary Harnack principle helped the authors come up with the idea in this paper. We thank Jing An for providing suggestions in writing the paper. The authors would thank Tarek Elgindi for constant support.
\section*{Clarification}
\subsection*{Data availability}
    Data sharing is not applicable to this article as no data sets were generated or analyzed
during the current study.
\subsection*{Conflicts of interest}
The authors have no conflict of interest to declare for this article.
\subsection*{Ethic approval}
Not applicable.
\subsection*{Funding}
This research received no specific grant from any funding agency in the public, commercial, or not-for-profit sectors.

\end{document}